\documentclass[pdflatex,sn-mathphys-num]{sn-jnl}

\usepackage{graphicx}%
\usepackage{wrapfig}
\usepackage{subfigure}
\usepackage{multirow}%
\usepackage{amsmath,amssymb,amsfonts}%
\usepackage{mathtools}
\usepackage{nicefrac}
\usepackage{amsthm}%
\usepackage{mathrsfs}%
\usepackage[title]{appendix}%
\usepackage{xcolor}%
\usepackage{textcomp}%
\usepackage{manyfoot}%
\usepackage{booktabs}%
\usepackage{algorithm}%
\usepackage{algorithmicx}%
\usepackage{algpseudocode}%
\usepackage{listings}%

\newcommand{\R}{\mathbb{R}}
\newcommand{\sY}{{\R^{d_y}}}
\newcommand{\sX}{{\R^{d_x}}}

\newcommand{\und}[2]{\underset{#2}{#1}}

\theoremstyle{thmstyleone}%
\newtheorem{theorem}{Theorem}
\theoremstyle{thmstyletwo}%

\theoremstyle{thmstylethree}%
\newtheorem{definition}{Definition}%

\theoremstyle{thmstylethree}
\newtheorem{lemma}{Lemma}

\begin{document}

\title[Strengthening the finite characterizations of smooth min-max games]{Strengthening the finite characterizations of smooth min-max games}


\author*[1]{\fnm{Valery} \sur{Krivchenko}}\email{krivchenko.vo@phystech.edu}

\author[2]{\fnm{Alexander} \sur{Gasnikov}}\email{gasnikov@yandex.ru}

\author[3]{\fnm{Dmitry} \sur{Kovalev}}\email{dakovalev1@gmail.com}

\affil*[1]{\orgdiv{Research Center of the Artificial Intelligence Institute}, \orgname{Innopolis University}, \orgaddress{\street{Universitetskaya St.}, \city{Innopolis}, \postcode{420500}, \state{Republic of Tatarstan}, \country{Russia}}}

\affil[2]{\orgname{Innopolis University}, \orgaddress{\street{Universitetskaya St.}, \city{Innopolis}, \postcode{420500}, \state{Republic of Tatarstan}, \country{Russia}}}

\affil[3]{\orgdiv{Yandex Research}, \orgname{Yandex}, \orgaddress{\street{Ulitsa L'va Tolstogo}, \city{Moscow}, \postcode{119021}, \state{Moscow}, \country{Russia}}}


\abstract{Saddle point problems with smooth convex–concave objective functions are often used to model  min-max problems arising in machine learning. First-order methods are the standard paradigm for solving such problems. Therefore, it is important to know how those methods behave in the worst-case scenarios. In order to derive the guarantees, one would require the inequalities that appropriately constrain the iterates, gradient's and function's values. In this paper, we present stronger constraints for smooth convex–concave functions and show that they could allow tighter upper bounds for first-order methods.}




\maketitle

\section{Introduction}\label{sec:Introduction} 

In this paper we consider the following saddle point problem (SPP):
\begin{align} \label{problem:minmax}
	\und{\min}{x \in \sX} \und{\max}{y \in \sY} f(x,y),
\end{align}
where $f$ is a smooth convex--concave function. In literature \cite{zhang2021unified}, this setting is known as smooth min–max games and has been a subject of active research mostly due to applications in machine learning. Notable examples include generative adversarial networks \cite{goodfellow2020generative} and reinforcement learning \cite{li2019robust}. 

First-order gradient-based methods are undoubtedly the standard paradigm for solving large-scale optimization problems. The main reason for this is the low cost per iteration and satisfactory effectiveness. Therefore, there is a need to understand how such methods perform in the worst possible scenarios within the defined setting and to compare their worst-case performance. When it comes to the exact worst-case behavior of first-order methods in min-max setting, there remains a lot of unknown \cite{zhang2021unified,zhang2022near,lee2024fundamental}. 

Almost all proofs of worst-case bounds and all proofs obtained via PEP \cite{drori2014performance, goujaud2023fundamental} involve constraints that describe algebraic relations between quantities appearing in the algorithm, most notably the iterates  $x_i$, the gradients $\nabla f(x_i)$ and the functional values $f(x_i)$. These constraints implicitly impose restrictions on the admissible gradients and functional values, ensuring that they are consistent with some function from function class. Ideally, the set of admissible iterates, gradients, and functional values should be restricted to those that can be \textit{interpolated} by a real function belonging to the considered class. Otherwise, the feasible set may contain artificial configurations that do not correspond to any actual function, leading to overly pessimistic bounds \cite{rubbens2023interpolation}.

Interpolation of function classes has been a subject of intensive research and has been particularly fruitful in the convex setting \cite{taylor2017smooth, rubbens2024constraint}. Those results enabled the exact solution of worst-case performance problems for a number of key first-order algorithms \cite{De_Klerk2017-fm, doi:10.1137/16M108104X, Taylor2018-gk}. In contrast, in the convex–concave setting, interpolation conditions have only been derived for the general non-smooth case and for smooth convex–concave functions with bilinear coupled structure \cite{Krivchenko_Gasnikov_Kovalev_2024}. Interpolation conditions for smooth convex–concave functions with unspecified structure have not yet been derived. 

In this paper, we present stronger constraints for smooth convex--concave functions. Using the framework for automatic generation of Lyapunov functions \cite{taylor2018lyapunov}, we show that our constraints may allow noticeably tighter guarantees for first-order methods in the setting.

\section{Preliminaries}

Consider a general variational inequality (VI) problem~\cite{Harker1990}
\begin{align}
    \text{find} \hspace{4pt} z_* \in \Omega \hspace{8pt} \text{s.t.} \hspace{8pt} \langle g^z(z_*), z-z_*\rangle \geqslant 0 \hspace{8pt} \text{for all} \hspace{4pt} z \in \Omega,
    \label{eq:var_ineq}
\end{align}
where $\Omega$ is a non-empty convex subset of $\R^d$ and $g^z : \R^d \rightarrow \R^d$ is a continuous mapping. Convex--concave SPP \eqref{problem:minmax}, where $f$ is smooth, is a specific instance \cite{kinderlehrer2000introduction} of \eqref{eq:var_ineq} with
\begin{align*}
    g^z(z) = [\nabla_x^{\top} f(z), -\nabla_y^{\top} f(z)]^{\top},
\end{align*}
where $z = [x^{\top}, y^{\top}]^{\top}$. In addition, because $f$ is convex--concave, any solution $z_* = [x_*^{\top}, y_*^{\top}]^{\top}$ of \eqref{eq:var_ineq} is a global Nash Equilibrium~\cite{VonNeumann+Morgenstern:1944}:
\begin{align}
    f(x_*,y) \leqslant f(x_*,y_*) \leqslant f(x,y_*) \quad \forall x \in \sX, y \in \sY.
\end{align}

\begin{definition} Let $f \in \mathcal S$. We introduce a mapping $g^z(z) = [g^x(z)^{\top}, -g^y(z)^{\top}]^{\top}$ where $g^x(z) \in \partial_x f(z)$, $g^y(z) \in \partial_y f(z)$, and $z = [x^{\top}, y^{\top}]^{\top}$. We call $g^z(z)$ an operator of $f$.
\end{definition}

\begin{definition} Let $f \in \mathcal S$. We say that $f$ has $\mu$-strongly monotone operator or $f$ is $\mu$-strongly-convex-$\mu$-strongly-concave ($\mu$-SCSC) if
\begin{equation}
    \langle g^z(w) - g^z(z),w-z \rangle \geqslant 
        \mu \|w - z\|_2^2
\end{equation}
for all $z, w \in \sX \times \sY$.
\label{fed:monotone}
\end{definition}

\begin{definition} Let $f \in \mathcal S$. We say that $f$ has $L$-Lipschitz operator or $f$ is $L$-smooth if $L > 0$ and
    \begin{equation}
        \|g^z(w) - g^z(z)\|_2 \leqslant 
        L \|w - z\|_2
    \end{equation}
    for all $z, w \in \sX \times \sY$.
\label{fed:smooth}
\end{definition}

\begin{definition} Let $f \in \mathcal S$. We say that $f$ has $1/L$-cocoercive operator if $L > 0$ and
\begin{equation}
    \langle g^z(w) - g^z(z), w-z \rangle \geqslant 
        \frac{1}{L} \|g^z(w) - g^z(z)\|_2^2
\end{equation}
for all $w, z \in \sX \times \sY$.
\label{fed:cocoercive}
\end{definition}

We shall use the following notation for function classes:
\begin{align*}
    & \mathcal F: \text{ convex, } \\
    & \mathcal F_{\mu, L}: \mu-\text{strongly convex, }\ L-\text{smooth,} \\
    & \mathcal S: \text{convex--concave,} \\
    & \mathcal S_{\mu}: \mu-\text{SCSC ,} \\
    & \mathcal S_{\mu, L}: \mu-\text{SCSC}, \ L-\text{smooth ,} \\
    & \mathcal S^{cc}_{\mu, L}: \mu-\text{SCSC}, \ 1/L-\text{cocoercive operator},
\end{align*}
where $0 < \mu < L$ in all definitions.

In the existing studies \cite{taylor2017smooth}, the interpolability is defined as follows:

\begin{definition}
\textit{Let I be an index set and consider a sequence $A = \left\{(x_i, g_i,f_i)\right\}_{i \in I}$ where $x_i, g_i \in \R^d$, $f_i \in \R$ for all $i \in I$. Consider a set of convex functions $\mathcal F$. Sequence $A$ is $\mathcal F$-interpolable if and only if there exists function $f \in \mathcal F$ such that $g_i \in \partial f(x_i)$, $f_i = f(x_i)$ for all $i \in I$.}
\label{fed:interpolation_convex}
\end{definition}

For convex--concave functions, we define interpolable sequences similarly. 

\begin{definition}
\textit{Let I be an index set and consider a sequence $\widetilde{A}_I = \left\{(x_i, y_i, g^x_i, g^y_i, f_i)\right\}_{i \in I}$ where $x_i, g^x_i \in \sX$, $y_i, g^y_i \in \sY$, $f_i \in \R$ for all $i \in I$. Consider a set of convex--concave functions $\mathcal S$. Sequence $\widetilde{A}_I$ is $\mathcal S$-interpolable if and only if there exists function $f \in \mathcal S$ such that $g^x_i \in \partial_x f(x_i, y_i)$, $g^y_i \in \partial_y f(x_i, y_i)$, $f_i = f(x_i, y_i)$ for all $i \in I$.}
\label{fed:interpolation}
\end{definition}

\section{Auxiliary results}

\subsection{Strongly convex strongly concave functions}

Interpolation conditions for non-smooth convex--concave functions.

\begin{theorem} [$\mathcal S$-interpolation \cite{Krivchenko_Gasnikov_Kovalev_2024}] 
Sequence $\widetilde{A}_I$ is $\mathcal S_\mu$-interpolable if and only if $\forall i,j \in I$:
\begin{equation}
    f_i \geqslant f_j + \langle g^x_j,x_i-x_j \rangle + \langle g^y_i, y_i-y_j\rangle.
\end{equation}
\label{teo:nonsmooth_interpolation}
\end{theorem}

\begin{lemma} Let $f: \sX \times \sY \to \R$. The following statements are equivalent $\forall z_0, z_1 \in \sX \times \sY$:
\begin{align*}
    &(1): f \in \mathcal S, \\
    &(2): f(z_1) \geqslant f(z_0) + \langle g^x_0,x_1-x_0 \rangle + \langle g^y_1, y_1-y_0 \rangle, \\
    &(3): \langle g^z(z_1) - g^z(z_0),z_1-z_0 \rangle \geqslant 0,
\end{align*}
where $g^x_0 \in  \partial_x f(z_0)$, $g^x_1 \in  \partial_x f(z_1)$, $g^y_0 \in  \partial_y f(z_0)$ and $g^y_1 \in  \partial_y f(z_1)$.
\label{lem:saddle_properties}
\end{lemma}

\begin{proof}
$[(1) \iff (2)]$: Follows from Theorem~\ref{teo:nonsmooth_interpolation}. \\
$[(2) \implies (3)]$: From $(2)$, we have
\begin{align*}
    &f(z_1) \geqslant f(z_0) + \langle g^x_0,x_1-x_0 \rangle + \langle g^y_1,y_1-y_0 \rangle, \\
    &f(z_0) \geqslant f(z_1) + \langle g^x_1, x_0-x_1 \rangle + \langle g^y_0,y_0-y_1\rangle. 
\end{align*}
By adding up the inequalities, we get $(3)$. \\
$[(3) \implies (1)]$: for any fixed $y$:
\begin{align*}
    \langle g^x(x_1,y) - g^x(x_0,y), x_1-x_0 \rangle \geqslant 0, 
\end{align*}
which is equivalent to $f(\cdot, y)$ being convex. The same way, $f(x, \cdot)$ is concave. It means $f \in \mathcal S$ by definition.
\end{proof}

\begin{lemma} We have $f \in \mathcal S_\mu$ if and only if $f - \frac{\mu}{2} \|x\|_2^2 + \frac{\mu}{2} \|y\|_2^2 \in \mathcal S$.
\label{lem:monotonicity_trasform}
\end{lemma}
\begin{proof}
    The statement directly follows from Lemma~\ref{lem:saddle_properties} and Definition~\ref{fed:monotone} of $\mu$-SCSC functions.
\end{proof}

The addition of $\frac{\mu}{2} \|x\|_2^2 + \frac{\mu}{2} \|y\|_2^2$ regularization naturally induces a pointwise transform.

\begin{lemma} Consider sequence $\widetilde{A}_I$. The following statements are equivalent
\begin{align*}
    &(1): \left\{(x_i,\ y_i,\ g^x_i,\ g^y_i,\ f_i)\right\}_{i \in I} \text{ is } \mathcal S_\mu-\text{interpolable},\\
    &(2): \left\{\left(x_i,\ y_i,\ g^x_i - \mu x_i,\ g^y_i + \mu y_i,\ f_i - \frac{\mu\|x_i\|_2^2}{2} + \frac{\mu\|y_i\|_2^2}{2}\right)\right\}_{i \in I} \text{ is } \mathcal S-\text{interpolable} . 
\end{align*}
\label{lem:monotonicity_pointwise}
\end{lemma}
\begin{proof}
The statement follows from Lemma~\ref{lem:monotonicity_trasform}.
\end{proof}

\begin{theorem} [$\mathcal S_\mu$-interpolation] Sequence $\widetilde{A}_I$ is $\mathcal S_\mu$-interpolable if and only if $\forall i,j \in I$:
\begin{equation}
    f_i \geqslant f_j + \langle g^x_j,x_i-x_j \rangle + \langle g^y_i, y_i-y_j\rangle + \frac{\mu \|x_i-x_j\|_2^2}{2} + \frac{\mu \|y_i-y_j\|_2^2}{2}.
\end{equation}
\label{teo:monotone_conditions}
\end{theorem}
\begin{proof}
By applying Theorem~\ref{teo:nonsmooth_interpolation} to the $\mathcal S$-interpolable sequence from Lemma~\ref{lem:monotonicity_pointwise} we prove the statement.
\end{proof}

\subsection{Difference of quadratic functions}

Consider a rather trivial class $\mathcal S^q_\mu$:

\begin{definition} We say that $f \in \mathcal S^q_\mu$ ($0 < \mu < \infty$) if $f \in \mathcal S$ and $\forall z_0,z_1 \in \sX \times \sY$, it satisfies
\begin{align}
    &\langle g^z(z_1) - g^z(z_0), z_1 -z_0 \rangle = \frac{1}{\mu}\|g^z(z_1)-g^z(z_0)\|^2_2, \label{eq:quad1} \\
    &\langle g^z(z_1) - g^z(z_0), z_1 - z_0 \rangle = \mu \|z_1-z_2\|^2_2. \label{eq:quad2}
\end{align}
\end{definition}

Class $\mathcal S^q_\mu$ is also an extreme case of $\mathcal S_{\mu, L}$ (and $\mathcal S^{cc}_{\mu, L} \subset \mathcal S_{\mu, L}$) when $\mu = L$ because
\begin{align*}
    \mu \|\Delta z \|^2 \leqslant \langle \Delta g, \Delta z \rangle \leqslant \|\Delta g\| \|\Delta z\| \leqslant \mu \|\Delta z \|^2
\end{align*}

\begin{lemma} Class $S^q_\mu$ is a class of all functions of the following form ($\mu > 0$):
\begin{equation}
    f(x,y) = \frac{\mu\|x\|_2^2}{2} - \frac{\mu\|y\|_2^2}{2} + \langle a,x\rangle + \langle b,y\rangle + c,
\label{eq:saddle_quadratic}
\end{equation}
where $a \in \sX$, $b \in \sY$ and $d \in \R$.
\label{lem:saddle_quadratic_class}
\end{lemma}
\begin{proof}
Let $f \in \mathcal S^q_\mu$. From \eqref{eq:quad1} and \eqref{eq:quad2}, we have
\begin{align}
    \|g^z(z_1) - g^z(z_0)\|_2 = \mu \|z_1-z_0\|_2
    \label{eq:quad3}
\end{align}
for all $z_0,z_1 \in \sX \times \sY$. We also have
\begin{align*}
    \langle g^z(z_1) - g^z(z_0), z_1 -z_0 \rangle = \mu \|z_1-z_0\|_2^2 = \|g^z(z_1) - g^z(z_0)\|_2 \|z_1-z_0\|_2.
\end{align*}
Therefore, $\Delta g^z \| \Delta z$ and
\begin{align*}
    g^z(z_1) - g^z(z_0) = \lambda(z_1,z_0) \left(z_1 - z_0\right)
\end{align*}
 where $\lambda = \mu$ due to \eqref{eq:quad3}. In other words,
 \begin{align*}
     \nabla_x f = \mu x + a, \quad \nabla_y f = -\mu y + b,
 \end{align*}
where $a \in \sX$, $b \in \sY$ are some vectors. By integrating, we get \eqref{eq:saddle_quadratic}. On the other hand, if $f$ is a function of the form \eqref{eq:saddle_quadratic} then $f \in \mathcal S^q_\mu$ from definition.
\end{proof}

To interpolate $\mathcal S^q_\mu$, we can use known \cite{rubbens2024constraint} interpolation conditions for $\mathcal F_{\mu,\mu}$ which is all functions of the form
\begin{equation*}
    f(x) = \frac{\mu\|x\|_2^2}{2} + \langle a, x \rangle + b,
\end{equation*}
where $a \in \R^d$, $b \in \R$. Any function from $\mathcal S^q_\mu$ is a difference of two functions from $\mathcal F_{\mu, \mu}$.

\begin{theorem} [Quadratic interpolation~\cite{rubbens2024constraint}] Sequence $\left\{(x_i, g^x_i, f_i)\right\}_{i \in I}$ is $\mathcal F_{\mu, \mu}$-interpolable if and only if $\forall i,j \in I$:
\begin{align*}
    & f_i = f_j + \frac{1}{2}\langle g^x_i + g^x_j,x_i-x_j\rangle,\\
    & g^x_i - g^x_j = \mu(x_i-x_j).
\end{align*}
\label{teo:quadratic_interpolation}
\end{theorem}
\begin{theorem} Sequence $\widetilde{A}_I$ is $\mathcal S^q_\mu$-interpolable if and only if $\forall i,j \in I$:
\begin{align}
    & f_i = f_j +\frac{1}{2}\langle g^x_i + g^x_j,x_i-x_j\rangle + \frac{1}{2}\langle g^y_i + g^y_j,y_i-y_j\rangle, \label{eq:quad_int1}\\
    & g^x_i-g^x_j = \mu(x_i-x_j), \label{eq:quad_int2}\\
    & g^y_i-g^y_j = -\mu(y_i-y_j). \label{eq:quad_int3}
\end{align}
\label{teo:saddle_quadratic_interpolation}
\end{theorem}
\begin{proof}
It's straightforward to verify that any $f \in \mathcal S^q_\mu$ satisfies the conditions which proves necessity. 

Next, we prove sufficiency. Suppose that $\widetilde{A}_I$ satisfies the conditions. Consider a system of linear equations
\begin{equation}
p_i - p_j = \frac{1}{2}\langle g^x_i + g^x_j,x_i-x_j\rangle, \quad \forall i,j \in I,
\label{eq:lin_eq_p}
\end{equation}
where $\{p_i\}_{i\in I}$ are variables. By substituting \eqref{eq:quad_int2} into \eqref{eq:lin_eq_p}, we get $\forall i,j \in I$:
\begin{align}
    p_i - p_j = \frac{1}{2\mu}\left(\|g^x_i\|_2^2 - \|g^x_j\|_2^2\right)
\end{align}
Thus, system \eqref{eq:lin_eq_p} has a family of solutions with one example being $p_i := \frac{1}{2\mu}\|g^x_i\|_2^2, \ \forall i\in I$. By Theorem~\ref{teo:quadratic_interpolation}, sequence $\left\{(x_i, g^x_i, p_i)\right\}_{i \in I}$ can be interpolated by some $p \in \mathcal F_{\mu, \mu}$. Next, by substituting \eqref{eq:lin_eq_p} into \eqref{eq:quad_int1}, we obtain $\forall i,j \in I$:
\begin{align}
    q_i = q_j - \frac{1}{2}\langle g^y_i + g^y_j,y_i-y_j\rangle
\label{eq:lin_eq_q}
\end{align}
where we introduced $\{q_i\}_{i\in I}$ such that $\forall i \in I: q_i := p_i - f_i$. By substituting \eqref{eq:quad_int3} into \eqref{eq:lin_eq_q} and using the same argument, we have that the sequence $\left\{(y_i, -g^y_i,\ q_i)\right\}_{i \in I}$ is interpolable by some $q \in \mathcal F_{\mu, \mu}$. 

Finally, we define $f(x,y) := p(x) - q(y)$. One can see that $f \in \mathcal S^q_\mu$ and $f$ correctly interpolates the original sequence $\widetilde{A}_I$.
\end{proof}

\subsection{Correspondence between functions with Lipschitz and cocoercive operators}

\begin{theorem} $f \in \mathcal S^{cc}_{\mu + L,\ 2L}$ if and only if $f - \frac{L}{2}\|x\|_2^2 + \frac{L}{2}\|y\|_2^2 \in \mathcal S_{\mu, L}$.
\label{teo:smooth_cocoercive}
\end{theorem}
\begin{proof}
Let $f \in \mathcal S^{cc}_{\mu + L,\ 2L}$. By definition it means that $\forall z_0, z_1 \in \sX \times \sY$:
\begin{align*}
    &\langle g_z(z_1)-g_z(z_0),z_1-z_0 \rangle \geqslant \frac{1}{2L}\|g_z(z_1)-g_z(z_0)\|_2^2, \\
    &\langle g_z(z_1)-g_z(z_0),z_1-z_0 \rangle \geqslant  \left(\mu + L\right)\|z_1 - z_0\|_2^2.
\end{align*}
Consider $\phi(x,y) := f(x,y) - \frac{L}{2}\|x\|_2^2 + \frac{L}{2}\|y\|_2^2$. Its operator $\overline{g_z}(z) = [\nabla_x \phi(z)^\top, -\nabla_y \phi(z)^\top]^\top$ satisfies
\begin{align*}
    &\langle \overline{g_z}(z_1)-\overline{g_z}(z_0),z_1-z_0\rangle  + L\|z_1-z_0\|_2^2 \geqslant \frac{1}{2L}\|\overline{g_z}(z_1)-\overline{g_z}(z_0) + L(z_1 - z_0)\|_2^2, \\
    &\langle \overline{g_z}(z_1)-\overline{g_z}(z_0),z_1-z_0 \rangle \geqslant \mu \|z_1 - z_0\|_2^2.
\end{align*}
After simplification of the first inequality
\begin{align*}
    &\langle \overline{g_z}(z_1)-\overline{g_z}(z_0),z_1-z_0\rangle  + L\|z_1-z_0 \|_2^2 \\&\quad \geqslant \frac{1}{2L}\|\overline{g_z}(z_1)-\overline{g_z}(z_0)\|_2^2 + \langle \overline{g_z}(z_1)-\overline{g_z}(z_0),z_1-z_0 \rangle + \frac{L}{2}\|z_1-z_0\|_2^2,
\end{align*}
the expression reduces to 
\begin{align*}
    \|\overline{g_z}(z_1)-\overline{g_z}(z_0)\|_2 \leqslant L\|z_1-z_0\|_2.
\end{align*}
Together with the second inequality, we have
\begin{align*}
    &\|\overline{g_z}(z_1)-\overline{g_z}(z_0)\|_2 \leqslant L\|z_1-z_0\|_2, \\
    &\langle \overline{g_z}(z_1)-\overline{g_z}(z_0), z_1-z_0\rangle \geqslant \mu \|z_1-z_0\|_2^2.
\end{align*}
Thus, $\phi \in \mathcal S_{\mu, L}$ by definition. Similarly, by letting $\phi \in \mathcal S_{\mu, L}$ and performing all steps in reverse, we prove $f \in \mathcal S^{cc}_{\mu + L,\ 2L}$.
\end{proof}

If we get inequalities for one of the classes, we can get inequalities for the other class with Theorem~\ref{teo:smooth_cocoercive}.

\section{Tighter constraints for smooth convex--concave functions}

Among all equivalent characterizations of a function class, the interpolation conditions are the strictest when imposed on a finite set of points \cite{rubbens2023interpolation,rubbens2024constraint}. From this perspective, we aim to construct the strictest possible characterizations for $\mathcal S_{\mu, L}$ and $\mathcal S^{cc}_{\mu, L}$. 

Let $f \in \mathcal S_{\mu, L}$. By combining the interpolation conditions for $\mathcal S_{\mu}$ (Theorem~\ref{teo:monotone_conditions}) with Definition~\ref{fed:smooth}, we get $\forall z_0, z_1 \in \sX \times \sY$:
\begin{align} \label{eq:improved_smooth}
     &0 \geqslant \|g^z(z_1)-g^z(z_0)\|^2_2 - L^2\|z_1-z_0\|^2_2, \\
    &d(z_1,z_0) \geqslant \frac{\mu\|z_1-z_0\|_2^2}{2}, \notag
\end{align}
where
\begin{equation*}
    d(z_1,z_0) := f(z_1) -  f(z_0) - \langle g^x(z_0),x_1-x_0\rangle - \langle g^y(z_1),y_1-y_0\rangle.
\end{equation*}
The constraints are not sufficient for interpolability because they do not recover interpolation conditions for smooth strongly convex functions \cite{taylor2017smooth}. Therefore, we derive additional pairwise constraints for $f \in \mathcal S_{\mu, L}$.

\begin{lemma} [Definition for Lipschitz continuous gradient \cite{Taylor2017ConvexIA}]. Let $f$ be a continuous function and $L > 0$. We say that $f$ is $L$-smooth if \ $\forall x, y \in \sX \times \sY$:
\begin{equation}
    \left|f(y) - f(x) - \langle \nabla f(x),y-x\rangle \right| \leqslant \frac{L\|y-x\|_2^2}{2}.
\end{equation}
\label{lem:smooth_saddle_alt_def}
\end{lemma}

For $f \in \mathcal S_{0,L}$, Lemma~\ref{lem:smooth_saddle_alt_def} implies $\forall z_0,z_1 \in \sX \times \sY$:
\begin{align*}
    &f(z_1) \leqslant f(z_0) + \langle g^x(z_0),x_1-x_0\rangle + \langle g^y(z_0),y_1-y_0\rangle + \frac{L}{2}\|z_1-z_0\|_2^2, \\
    &f(z_1) \geqslant f(z_0) + \langle g^x(z_0),x_1-x_0\rangle + \langle g^y(z_0),y_1-y_0\rangle - \frac{L}{2}\|z_1-z_0\|_2^2.
\end{align*}

\begin{theorem} Let $f \in \mathcal S_{\mu, L}$. Then $\forall z_0, z_1 \in \sX \times \sY$:
\begin{align*}
    f(z_1) &\geqslant f(z_0) + \langle g^x(z_0),x_1-x_0\rangle + \langle g^y(z_1),y_1-y_0\rangle \\&\quad+ \frac{1}{2(L-\mu)}\|g^x(z_1) - g^x(z_0) - \mu(x_1-x_0)\|_2^2 + \frac{\mu}{2}\|x_1-x_0\|_2^2 - \frac{L}{2}\|y_1-y_0\|_2^2, \\
   f(z_1) &\geqslant f(z_0) + \langle g^x(z_0), x_1-x_0\rangle + \langle g^y(z_1),y_1-y_0\rangle \\&\quad+ \frac{1}{2(L-\mu)}\|g^y(z_1) - g^y(z_0) + \mu(y_1-y_0)\|_2^2 + \frac{\mu}{2}\|y_1-y_0\|_2^2 - \frac{L}{2}\|x_1-x_0\|_2^2.
\end{align*}
\label{lem:smooth_monotone_saddle_property}
\end{theorem}
\begin{proof}
From Lemma~\ref{lem:smooth_saddle_alt_def}:
\begin{equation}
    f(z_1) \leqslant f(z_0) + \langle g^x(z_0),x_1-x_0\rangle + \langle g^y(z_0),y_1-y_0\rangle + \frac{L}{2}\|z_1-z_0\|_2^2.
    \label{eq:smsp_1}
\end{equation}
Consider a new function
\begin{align*}
    \phi = f - \frac{\mu\|x\|_2^2}{2}.
\end{align*}    
We also denote
\begin{align*}
    &\phi^x := \nabla_x \phi(z) = g^x(z) - \mu x,\\
    &\phi^y := \nabla_y \phi(z) = g^y(z).
\end{align*}
We can rewrite \eqref{eq:smsp_1} in terms of $\phi$:
\begin{align}
    \phi(z_1) &\leqslant \phi(z_0) + \langle \phi^x(z_0),x_1 -x_0\rangle  \notag\\&\quad+ \langle \phi^y(z_0),y_1 -y_0\rangle + \frac{L}{2}\|y_1-y_0\|_2^2 + \frac{(L-\mu)}{2}\|x_1-x_0\|_2^2.
    \label{eq:smsp_2}
\end{align}
Since $f \in \mathcal S_{\mu, L}$, it means $\phi \in \mathcal S$. We also introduce $\theta \in \mathcal S$:
\begin{align*}
    \theta(z) = \phi(z) - \langle \phi^x(z_\ast),x-x_\ast\rangle,
\end{align*}
and denote
\begin{align*}
    &\theta^x := \nabla_x \theta(z) = \phi^x(z) - \phi^x(z_\ast), \\
    &\theta^y := \nabla_y \theta(z) = \phi^y(z),
\end{align*}
where $z_\ast = [x_\ast^\top, y_\ast^\top]^\top$ some arbitrary point. Because $\phi \in \mathcal S$, we have
\begin{align*}
    x_\ast = \und{\arg\min}{x} \hspace{2pt} \theta(x,y_\ast).
\end{align*}
From \eqref{eq:smsp_2}, we have
\begin{align}
    \theta(z_1) &\leqslant \theta(z_0) + \langle \theta^x(z_0),x_1 -x_0\rangle \notag\\ &\quad+ \langle \theta^y(z_0),y_1 -y_0\rangle + \frac{L}{2}\|y_1-y_0\|_2^2 + \frac{(L-\mu)}{2}\|x_1-x_0\|_2^2.
    \label{eq:smsp_3}
\end{align}
Next, we write an upper bound for $\theta(z_\ast)$:
\begin{align*}
    &\theta(x_\ast,y_\ast) \\&\quad = \und{\min}{x} \ \theta(x,y_\ast) \\
    &\quad \leqslant \und{\min}{x} \left [\theta(x_1,y_1) + \langle\theta^x(x_1,y_1),x -x_1\rangle + \langle \theta^y(x_1,y_1),y_\ast -y_1\rangle \right.\\&\quad \quad \left. + \frac{L}{2}\|y_\ast -y_1\|_2^2 +\frac{(L-\mu)}{2}\|x -x_1\|_2^2 \right] \\
    &\quad = \und{\min}{x} \left [\theta(x_1,y_1) + \langle \theta^y(x_1,y_1),y_\ast -y_1\rangle + \frac{L}{2}\|y_\ast -y_1\|_2^2 \right. \\&\quad \quad \left.- \|\theta^x(x_1,y_1)\|_2 \|x-x_1\|_2 + \frac{(L-\mu)}{2}\|x -x_1\|_2^2 \right] \\
    &\quad = \theta(x_1,y_1) + \langle \theta^y(x_1,y_1),y_\ast -y_1\rangle + \frac{L}{2}\|y_\ast -y_1\|_2^2 - \frac{\|\theta^x(x_1,y_1)\|_2^2}{2(L-\mu)}.
\end{align*}
Rewriting the above inequality in terms of $\phi$:
\begin{align*}
    \phi(x_1,y_1) &\geqslant \phi(x_\ast,y_\ast) + \langle \phi^x(x_\ast,y_\ast),x_1-x_\ast\rangle  \\
    &\quad + \langle \phi^y(x_1,y_1),y_1 -y_\ast\rangle + \frac{\|\phi^x(x_1,y_1)-\phi^x(x_\ast,y_\ast)\|_2^2}{2(L-\mu)} - \frac{L}{2}\|y_1 -y_\ast\|_2^2,
\end{align*}
and then, in terms of the original function $f \in \mathcal S_{\mu, L}$:
\begin{align*}
    f(z_1) &\geqslant f(z_0) +\langle  g^x(z_0),x_1-x_0\rangle + \langle g^y(z_1),y_1-y_0\rangle  \\&\quad+ \frac{1}{2(L-\mu)}\|g^x(z_1) - g^x(z_0) - \mu(x_1-x_0)\|_2^2 + \frac{\mu}{2}\|x_1-x_0\|_2^2 - \frac{L}{2}\|y_1-y_0\|_2^2, 
\end{align*}
where we changed $z_\ast \rightarrow z_0$.
Thus, we derived the first inequality. The second inequality is derived similarly, except we consider $\phi \in \mathcal S$:
\begin{align*}
    &\phi = f + \frac{\mu\|y\|_2^2}{2}, \\
    &\phi^x := \nabla_x \phi(z) = g^x(z),  \\
    &\phi^y := \nabla_y \phi(z) = g^y(z) + \mu y,
\end{align*}
and $\theta \in \mathcal S$:
\begin{align*}
    &\theta(z) = \phi(z) - \langle \phi^y(z_\ast),y-y_\ast \rangle,\\
    &\theta^x := \nabla_x \theta(z) = \phi^x(z), \\
    &\theta^y := \nabla_y \theta(z) = \phi^y(z) - \phi^y(z_\ast).
\end{align*}
Finally, instead of minimization over $x$ we maximize over $y$ to get a lower bound for $\theta(x_\ast,y_\ast)$.
\end{proof}

\begin{theorem} If  $f \in \mathcal S_{\mu, L}$ then $\forall z_1, z_0 \in \sX \times \sY$ the following inequalities hold:
\begin{align}
        0 &\geqslant \|g^z(z_1)-g^z(z_0)\|^2_2 - L^2\|z_1-z_0\|^2_2, \label{eq:lem_smooth_1}\\
        d(z_1,z_0) &\geqslant \frac{\mu\|z_1-z_0\|_2^2}{2} \label{eq:lem_smooth_2},\\
        d(z_1,z_0) &\geqslant \frac{1}{2(L-\mu)}\|g^x(z_1) - g^x(z_0) - \mu(x_1-x_0)\|_2^2 + \frac{\mu}{2}\|x_1-x_0\|_2^2 \notag \\&\quad- \frac{L}{2}\|y_1-y_0\|_2^2 \label{eq:lem_smooth_3},\\
        d(z_1,z_0) &\geqslant \frac{1}{2(L-\mu)}\|g^y(z_1) - g^y(z_0) + \mu(y_1-y_0)\|_2^2 + \frac{\mu}{2}\|y_1-y_0\|_2^2 \notag\\&\quad- \frac{L}{2}\|x_1-x_0\|_2^2 \label{eq:lem_smooth_4},
\end{align}
where
\begin{align*}
    d(z_1,z_0) := f(z_1) -  f(z_0) - \langle g^x(z_0),x_1-x_0\rangle - \langle g^y(z_1),y_1-y_0\rangle.
\end{align*}
\label{teo:smooth_monotone_interpolation_candidate}
\end{theorem}
\begin{proof}
The first pair of constraints is \eqref{eq:improved_smooth}. The last two inequalities come from Theorem~\ref{lem:smooth_monotone_saddle_property}.
\end{proof}

Although the last two inequalities in Theorem~\ref{teo:smooth_monotone_interpolation_candidate} are redundant when imposed on all points, they strengthen the first two when imposed on the sequence $\widetilde{A}_I$. With Theorem~\ref{teo:smooth_cocoercive}, we automatically obtain inequalities for $\mathcal S^{cc}_{\mu, L}$.

\begin{theorem} If $f \in \mathcal S^{cc}_{\mu, L}$ then  $\forall z_1, z_0 \in \sX \times \sY$ the following inequalities hold:
\begin{align}
        0 &\geqslant \frac{1}{L}\|g^z(z_1)-g^z(z_0)\|^2_2 - \langle g^z(z_1)-g^z(z_0), \ z_1-z_0 \rangle,  \\
        d(z_1,z_0) &\geqslant \frac{\mu\|z_1-z_0\|_2^2}{2}, \\
        d(z_1,z_0) &\geqslant \frac{1}{2(L-\mu)}\|g^x(z_1) - g^x(z_0) - \mu(x_1-x_0)\|_2^2 + \frac{\mu}{2}\|x_1-x_0\|_2^2, \\
        d(z_1,z_0) &\geqslant \frac{1}{2(L-\mu)}\|g^y(z_1) - g^y(z_0) + \mu(y_1-y_0)\|_2^2 + \frac{\mu}{2}\|y_1-y_0\|_2^2. 
\end{align}
where $d(z_1,z_0)$ defined the same as in Theorem~\ref{teo:smooth_monotone_interpolation_candidate}
\label{teo:cocoercive_monotone_interpolation_candidate}
\end{theorem}
\begin{proof}
Let $f \in \mathcal S^{cc}_{\mu+L,\ 2L}$ and $\phi := f - \frac{L}{2}\|x\|_2^2 + \frac{L}{2}\|y\|_2^2$. Then by Theorem~\ref{teo:smooth_cocoercive}, $\phi \in \mathcal S_{\mu,\ L}$ and satisfies the constraints from Theorem~\ref{teo:cocoercive_monotone_interpolation_candidate}. By substituting
\begin{align*}
    &\phi := f - \frac{L}{2}\|x\|_2^2 + \frac{L}{2}\|y\|_2^2, \\
    &\phi^x = g^x - L x, \quad \phi^y = g^y + L y,
\end{align*}
we derive the constraints.
\end{proof}

With Theorem~\ref{teo:smooth_monotone_interpolation_candidate} and Theorem~\ref{teo:cocoercive_monotone_interpolation_candidate} we provide the strongest inequalities for the respective classes, to the best of our knowledge.

\section{Experiment: showcasing the novel constraints}

\subsection{Description and results}

We show that the novel constraints impose stronger restrictions on the iterates than weaker alternative descriptions which may lead to noticeably tighter upper bounds for algorithms. We consider the two variations of the baseline GDA scheme. For certification of linear convergence, we use the framework for automatic generation of quadratic Lyapunov functions \cite{taylor2018lyapunov} which combines the ideas of PEP and IQC \cite{lessard2016analysis}. All SDP problems were solved with MOSEK~\cite{mosek}.

The first set is inequalities for $\mu$-strongly monotone, $L$-Lipshitz operator. For all $i,j \in I$:
\begin{align}
    & \| g^z(z_i)-g^z(z_j)\|_2 \leqslant  L \| z_i-z_j \|_2, \notag \\
    &\langle g^z(z_i)-g^z(z_j), z_i-z_j\rangle \geqslant \mu\|z_i-z_j\|_2^2, \label{eq:definition_constraints}
\end{align}
The second set is the first two constraints from Theorem~\ref{teo:smooth_monotone_interpolation_candidate}:
\begin{align}
        & \| g^z(z_i)-g^z(z_j)\|_2 \leqslant  L \| z_i-z_j \|_2 \notag, \\
        &d_{ij} \geqslant \frac{\mu\|z_i-z_j\|_2^2}{2}. \label{eq:reduced_constraints}
\end{align}
And finally, the full set of constraints from Theorem~\ref{teo:smooth_monotone_interpolation_candidate}. 

\begin{minipage}[htp]{0.48\textwidth} 
    \begin{algorithm}[H]
    \caption{Sim-GDA.}
    	\label{alg:sim}
    \begin{algorithmic}[1]
    \State 
    \noindent {\bf Input:} Starting point $z_{start}$, number of steps $N$.
    \For {$k = 1, \dots, N$}
            \State $x_{k+1} = x_{k} - \eta \nabla_x f(x_k,y_k),$
            \State $y_{k+1} = y_{k} + \eta \nabla_y f(x_{k},y_k).$
    \EndFor
    \State \noindent {\bf Output:} $z_N.$
    \end{algorithmic}
    \end{algorithm}
\end{minipage} \hfill
\begin{minipage}[htp]{0.48\textwidth} 
    \begin{algorithm}[H]
    \caption{Alt-GDA.}
    	\label{alg:alt}
    \begin{algorithmic}[1]
    \State 
    \noindent {\bf Input:} Starting point $z_{start}$, number of steps $N$.
    \For {$k = 1, \dots, N$}
            \State $x_{k+1} = x_{k} - \eta \nabla_x f(x_k,y_k),$
            \State $y_{k+1} = y_{k} + \eta \nabla_y f(x_{k+1},y_k).$
    \EndFor
    \State \noindent {\bf Output:} $z_N.$
    \end{algorithmic}
    \end{algorithm}
\end{minipage}

For Sim-GDA, there is no improvement (Fig.~\ref{fig:Fig1}) over existing results~\cite{zhang2021unified}, except when the condition number $\kappa = \frac{L}{\mu}$ is close to $1$.

For Alt-GDA (Figs. \ref{fig:Fig3}, \ref{fig:Fig4}, \ref{fig:Fig5}) we managed to improve the current best numerical estimate of complexity which is $\widetilde{\mathcal{O}}\left(\kappa^{1.385}\right)$. It was provided by \cite{lee2024fundamental} for $\mathcal S_{\mu,\mu, L, L, L}$ class using the branch-and-bound PEP framework (BnB-PEP) \cite{Gupta_2023}. Using the novel constraints, we obtained the estimate of $\widetilde{\mathcal{O}}\left(\kappa^{1.26}\right)$ for $\mathcal S_{\mu, L}$ that also holds for $\mathcal S_{\mu,\mu, L, L, L} \subset \mathcal S_{\mu, 2L}$.  The graph of iteration complexity $N(\kappa)$ is depicted on Fig.~\ref{fig:Fig6} where $N = -1/\log \rho$ is the number of iterations to convergence. We also manage to guarantee linear convergence over a wider range of step sizes. The gap becomes even larger for big condition numbers (Fig.~\ref{fig:Fig5}). 

Using the weaker constraints, we could only verify (roughly) $\widetilde{\mathcal{O}}\left(\kappa^{1.34}\right)$ which is closer to the result from \cite{lee2024fundamental}.

\begin{figure}[!ht]
\centering
\includegraphics[width=0.8\linewidth]{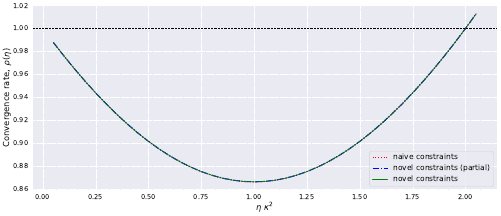} 
\caption{Sim-GDA: worst case convergence rate $\rho$ as a function of the step size $\eta$, $\kappa$ = 2.} \label{fig:Fig1}

 \vspace{\floatsep}
\centering
\includegraphics[width=0.8\linewidth]{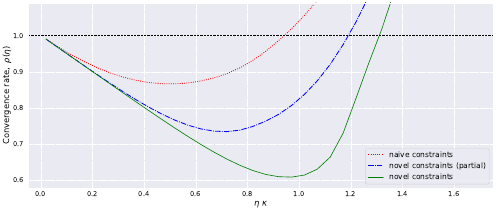} 
\caption{Alt-GDA: worst case convergence rate $\rho$ as a function of the step size $\eta$, $\kappa$ = 2.} \label{fig:Fig2}

 \vspace{\floatsep}
\centering
\includegraphics[width=0.8\linewidth]{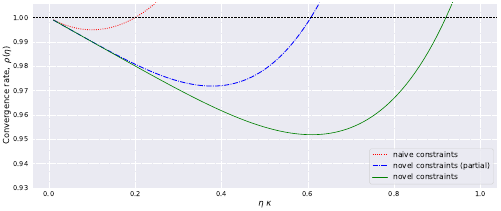} 
\caption{Alt-GDA: worst case convergence rate $\rho$ as a function of the step size $\eta$, $\kappa$ = 10.} \label{fig:Fig3}
\end{figure}

\begin{figure}[!ht]
\centering
\includegraphics[width=0.8\linewidth]{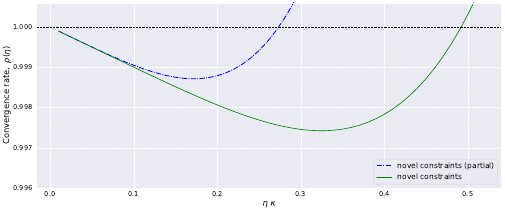} 
\caption{Alt-GDA: worst case convergence rate $\rho$ as a function of the step size $\eta$, $\kappa$ = 100.} \label{fig:Fig4}

\vspace{\floatsep}
\centering
\includegraphics[width=0.8\linewidth]{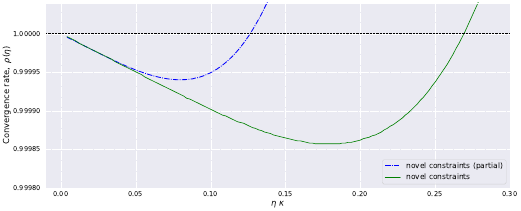} 
\caption{Alt-GDA: worst case convergence rate $\rho$ as a function of the step size $\eta$, $\kappa$ = 1000.} \label{fig:Fig5}

 \vspace{\floatsep}
\centering
\includegraphics[width=0.8\linewidth]{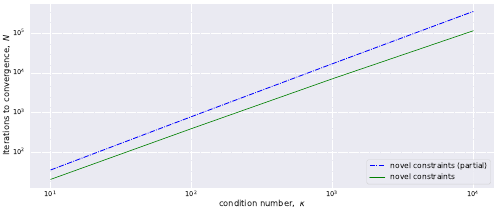} 
\caption{Alt-GDA: iterations to convergence $N$ as a function of the condition number.} \label{fig:Fig6}
\end{figure}

\clearpage

\subsection{SDP program for the experiment} 

The problem of linear convergence certificatian via quadratic Lyapunov function can be formulated as a convex semidefinite feasibility problem~\cite{taylor2018lyapunov}. 

\begin{align*} \label{problem:Lyapunov}\tag{Lyap}
    \textbf{feasible} \hspace{4pt} &(\{\lambda_{ij}^k\}_{i,j\in I_0}^{c \in C}, \ \{\nu_{ij}^c\}_{i,j\in I_1}^{c \in C},\ P_x,\ P_y )  \\
    & \lambda_{ij}^c  \geqslant 0 \quad \forall i,j \in I_0 ; \hspace{4pt} \forall c \in C,\\
    & \nu_{ij}^c \geqslant 0 \quad \forall i,j \in I_1 ; \hspace{4pt} \forall c \in C,\\
    &P^{(0)}_x -\underset{i,j\in I_0; \ c \in C}{\sum} \lambda_{ij}^c A^{c (0)}_{ij} \succeq 0, \\
    &P^{(0)}_y -\underset{i,j\in I_0;\ c \in C}{\sum} 
    \lambda_{ij}^c B^{c (0)}_{ij} \succeq 0, \\
    &\underset{i,j\in I_0;\ c \in C}{\sum} 
    \lambda_{ij}^c m^{c(0)}_{ij} = 0, \\
    &P^{(1)}_x +\underset{i,j\in I_1; c \in C}{\sum} \nu_{ij}^c A^{c(1)}_{ij} \preceq 0, \\
    &P^{(1)}_y + \underset{i,j\in I_1; c \in C}{\sum} 
    \nu_{ij}^c B^{c(1)}_{ij} \preceq 0, \\
    &\underset{i,j\in I_1; c \in C}{\sum} 
    \nu_{ij}^c m^{c(1)}_{ij} = 0.
\end{align*}

As an example, we show how parameters are defined for Sim-GDA . First, we introduce index sets: $I_0 = \{0,\ast\}$, $I_1 = \{0,1,\ast\}$, $C = \{1,2,3,4\}$. Suppose, we iterate Sim-GDA for $K \in \{0,1\}$ steps. We introduce $W_x \in \R^{d_x\times(K+2)}$ and $W_y \in \R^{d_y\times(K+2)}$:
\begin{align*}
    &W_x := \left[x_0,g^x_0,\ldots,g^x_K \right], \quad W_y = \left[y_0,g^y_0\ldots,g^y_K \right], \\
    &f := [f_0,\ldots,f_K]
\end{align*}
and Gram matrices $G_x := W_x^\top W_x$, $G_y := W_y^\top W_y$. Next, we introduce row basis vectors that satisfy $\forall i \in I_K$:
\begin{align*}
    &x_i = W_x \left(\overline{x}^{(K)}_i\right)^\top, \quad g^x_i = W_x \left(\overline{g}^{x(K)}_i\right)^\top, \quad y_i = W_y \left(\overline{y}^{(K)}_i\right)^\top, \quad g^y_i = W_y \left(\overline{g}^{y(K)}_i\right)^\top, \\
    &f_i = f \left(\overline{f}^{(K)}_i\right)^\top.
\end{align*}
 The matrices $A^{c(K)}_{ij}$ and $B^{c(K)}_{ij}$ are used to define the functional constraints for $\mathcal S_{\mu, L}$:
\begin{align*}
    m_{ij}^{c(K)}\overline{f} + \text{tr}(G_x A^{c(K)}_{ij}) + \text{tr}(G_y B^{c(K)}_{ij}) \geqslant 0, \hspace{8pt} i,j \in I_K ; \hspace{8pt} c \in C,
\end{align*}
where 
\begin{align*}
    A^{c(K)}_{ij}  := 
\begin{bmatrix} 
\overline{x}_i \\ \overline{x}_j \\ \overline{g}^x_i \\ \overline{g}^x_j
\end{bmatrix}^\top
\!\!\!A^c
\begin{bmatrix} 
\overline{x}_i \\ \overline{x}_j \\ \overline{g}^x_i \\ \overline{g}^x_j
\end{bmatrix}
,\hspace{16pt} 
B^{c(K)}_{ij} : = 
\begin{bmatrix} 
\overline{y}_i \\ \overline{y}_j \\ \overline{g}^y_i \\ \overline{g}^y_j
\end{bmatrix}^\top
\!\!\!B^c 
\begin{bmatrix} 
\overline{y}_i \\ \overline{y}_j \\ \overline{g}^y_i \\ \overline{g}^y_j
\end{bmatrix},
\end{align*}
where
\begin{align*}
A^{1} = B^{1} : = \frac{1}{4} \begin{bmatrix*}
	1  & -1 & -1 & -1 \\
	-1 & 1 & 1 & 1 \\
	-1  & 1 & -1 & 1 \\
	-1 & 1 & 1 & -1 
    \end{bmatrix*}
\end{align*}
correspond to \eqref{eq:lem_smooth_1},
\begin{align*}
&A^{2} : = \frac{1}{2} \begin{bmatrix*}
	-\mu/L  & \mu/L & 0 & -1 \\
	\mu/L & -\mu/L & 0 & 1 \\
	0  & 0 & 0 & 0 \\
	-1 & 1 & 0 & 0 
    \end{bmatrix*},
\hspace{10pt}
B^{2} : = \frac{1}{2} \begin{bmatrix*}
	-\mu/L  & \mu/L & -1 & 0 \\
	\mu/L & -\mu/L & 1 & 0 \\
	-1  & 1 & 0 & 0 \\
	0 & 0 & 0 & 0 
    \end{bmatrix*}
\\
\end{align*}
correspond to \eqref{eq:lem_smooth_2},
\begin{align*}
&A^{3} : = \frac{1}{2(1-\mu/L)} \begin{bmatrix*}
	-\mu/L  & \mu/L & \mu/L & -1 \\
	\mu/L & -\mu/L & -\mu/L & 1 \\
	\mu/L  & -\mu/L & -1 & 1 \\
	-1 & 1 & 1 & -1 
    \end{bmatrix*},
\hspace{10pt}
B^{3} : = \frac{1}{2} \begin{bmatrix*}
	1  & -1 & -1 & 0 \\
	-1 & 1 & 1 & 0 \\
	-1  & 1 & 0 & 0 \\
	0 & 0 & 0 & 0 
    \end{bmatrix*}
\\
\end{align*}
correspond to \eqref{eq:lem_smooth_3},
\begin{align*}
&A^{4} : = \frac{1}{2} \begin{bmatrix*}
	1  & -1 & 0 & -1 \\
	-1 & 1 & 0 & 1 \\
	0  & 0 & 0 & 0 \\
	-1 & 1 & 0 & 0 
    \end{bmatrix*},
\hspace{10pt}
B^{4} := \frac{1}{2(1-\mu/L)} \begin{bmatrix*}
	-\mu/L  & \mu/L & -1 & \mu/L \\
	\mu/L & -\mu/L & 1 & -\mu/L \\
	-1  & 1 & -1 & 1 \\
	\mu/L & -\mu/L & 1 & -1 
    \end{bmatrix*}
\end{align*}	\\
correspond to \eqref{eq:lem_smooth_4}.
Finally, $m_{ij}^{c(K)}$ is the same $\forall c \in C$:
\begin{align*}
    m_{ij}^{c(K)} : = \overline{f_i} - \overline{f_j}.
\end{align*}
And lastly, the parts corresponding to quadratic Lyapunov function are defined as follows:
\begin{align*}
P^{(0)}_{x} : = 
\begin{bmatrix} 
\overline{x}_0 \\ \overline{g}^x_0 
\end{bmatrix}^\top
\!\!\!P_x 
\begin{bmatrix} 
\overline{x}_0 \\ \overline{g}^x_0 
\end{bmatrix}
,\hspace{16pt} 
P^{(0)}_{y} : = 
\begin{bmatrix} 
\overline{y}_0  \\ \overline{g}^y_0
\end{bmatrix}^\top
\!\!\!P_y 
\begin{bmatrix} 
\overline{y}_0  \\ \overline{g}^y_0
\end{bmatrix},
\end{align*}

\begin{align*}
P^{(1)}_{x} : = 
\begin{bmatrix} 
\overline{x}_1 \\ \overline{g}^x_1 
\end{bmatrix}^\top
\!\!\!P_x
\begin{bmatrix} 
\overline{x}_1 \\ \overline{g}^x_1 
\end{bmatrix}
- \rho^2
\begin{bmatrix} 
\overline{x}_0 \\ \overline{g}^x_0 
\end{bmatrix}^\top
\!\!\!P_x
\begin{bmatrix} 
\overline{x}_0 \\ \overline{g}^x_0 
\end{bmatrix},
\end{align*}

\begin{align*}
P^{(1)}_{y} : = 
\begin{bmatrix} 
\overline{y}_1 \\ \overline{g}^y_1 
\end{bmatrix}^\top
\!\!\!P_y
\begin{bmatrix} 
\overline{y}_1 \\ \overline{g}^y_1 
\end{bmatrix}
- \rho^2
\begin{bmatrix} 
\overline{y}_0 \\ \overline{g}^y_0 
\end{bmatrix}^\top
\!\!\!P_y
\begin{bmatrix} 
\overline{y}_0 \\ \overline{g}^y_0  
\end{bmatrix}.
\end{align*}
Parameter $\rho^2$ is a contraction factor of Lyapunov function and symmetric $P_x, P_y \in \R^{2\times2}$ are variables along with $\lambda_{ij}^c, \nu_{ij}^c \in \R$. Feasibility of Problem \eqref{problem:Lyapunov} guarantees existence of a quadratic Lyapunov function that decreases at the rate $\rho^2$ \cite{taylor2018lyapunov}. To find the smallest $\rho$, we perform a binary search. We also utilize scaling and set $L = 1$. 

The Problem \eqref{problem:Lyapunov} presented above corresponds to Sim-GDA, so for Alt-GDA we make the appropriate changes. For instance, we assume a Lyapunov function to be a quadratic form of coordinates and gradients at both $(x_i,y_i)$ and $(x_{i+1}, y_i)$.

Because the problem is solved for three different sets of inequalities, we impose additional constraints to enforce a specific set. For the full set, we let $\forall i,j$: $\lambda_{ij}^1 = \lambda_{ji}^1$,\  $\nu_{ij}^1 = \nu_{ji}^1$. For \eqref{eq:reduced_constraints}, we additionally set $\lambda_{ij}^c = 0$ and $\nu_{ij}^c = 0$ for $c \in  \{3, 4\}$ and $\forall i,j$. Finally, to obtain \eqref{eq:definition_constraints}, we also set $\forall i,j$: $\lambda_{ij}^2 = \lambda_{ji}^2$,\  $\nu_{ij}^2 = \nu_{ji}^2$.

\section{Conclusion}

We obtained stronger constraints for smooth SCSC functions. By utilizing the correspondence between functions with Lipschitz and cocoercive operators, we also obtain inequalities for SCSC functions with cocoercive operator. We show numerically that the new inequalities allow tighter upper bounds for first-order methods using using PEP-like framework.

\backmatter

\bibliography{sn-bibliography}

@article{taylor2017smooth,
  title     = "Smooth strongly convex interpolation and exact worst-case
               performance of first-order methods",
  author    = "Taylor, Adrien B and Hendrickx, Julien M and Glineur, Fran{\c
               c}ois",
  journal   = "Math. Program.",
  publisher = "Springer Science and Business Media LLC",
  volume    =  161,
  number    = "1-2",
  pages     = "307--345",
  month     =  jan,
  year      =  2017,
  language  = "en",
  doi = "10.1007/s10107-016-1009-3"
}

@article{drori2014performance,
  title     = "Performance of first-order methods for smooth convex
               minimization: a novel approach",
  author    = "Drori, Yoel and Teboulle, Marc",
  journal   = "Math. Program.",
  publisher = "Springer Science and Business Media LLC",
  volume    =  145,
  number    = "1-2",
  pages     = "451--482",
  month     =  jun,
  year      =  2014,
  language  = "en",
  doi = "10.1007/s10107-013-0653-0"
}

@article{lessard2016analysis,
  title     = "Analysis and design of optimization algorithms via integral
               quadratic constraints",
  author    = "Lessard, Laurent and Recht, Benjamin and Packard, Andrew",
  journal   = "SIAM J. Optim.",
  publisher = "Society for Industrial \& Applied Mathematics (SIAM)",
  volume    =  26,
  number    =  1,
  pages     = "57--95",
  month     =  jan,
  year      =  2016,
  language  = "en",
  doi = "10.1137/15M1009597"
}

@inproceedings{taylor2018lyapunov,
  title={Lyapunov functions for first-order methods: Tight automated convergence guarantees},
  author={Taylor, Adrien and Van Scoy, Bryan and Lessard, Laurent},
  booktitle={International Conference on Machine Learning},
  pages={4897--4906},
  year={2018},
  organization={PMLR}
}

@inproceedings{rubbens2023interpolation,
  title={Interpolation constraints for computing worst-case bounds in performance estimation problems},
  author={Rubbens, Anne and Bousselmi, Nizar and Colla, S{\'e}bastien and Hendrickx, Julien M},
  booktitle={2023 62nd IEEE Conference on Decision and Control (CDC)},
  pages={3015--3022},
  year={2023},
  organization={IEEE}
}

@inproceedings{lee2024fundamental,
author = {Lee, Jaewook and Cho, Hanseul and Yun, Chulhee},
title = {Fundamental benefit of alternating updates in minimax optimization},
year = {2024},
publisher = {JMLR.org},
booktitle = {Proceedings of the 41st International Conference on Machine Learning},
articleno = {1056},
numpages = {76},
address = {Vienna, Austria},
series = {ICML'24}
}

@article{zhang2021unified,
author = {Zhang, Guodong and Bao, Xuchan and Lessard, Laurent and Grosse, Roger},
title = {A unified analysis of first-order methods for smooth games via integral quadratic constraints},
year = {2021},
issue_date = {January 2021},
publisher = {JMLR.org},
volume = {22},
number = {1},
issn = {1532-4435},
journal = {J. Mach. Learn. Res.},
month = jan,
articleno = {103},
numpages = {39}
}

@inproceedings{zhang2022near,
  title={Near-optimal local convergence of alternating gradient descent-ascent for minimax optimization},
  author={Zhang, Guodong and Wang, Yuanhao and Lessard, Laurent and Grosse, Roger B},
  booktitle={International Conference on Artificial Intelligence and Statistics},
  pages={7659--7679},
  year={2022},
  organization={PMLR}
}

@misc{rubbens2024constraint,
  title        = "A constraint-based approach to function interpolation, with
                  application to performance estimation for weakly convex
                  optimisation",
  author       = "Rubbens, Anne and Hendrickx, Julien M",
  year         =  2024,
  archivePrefix={arXiv},
  primaryClass = "math.OC",
  eprint       = "2405.08405",
  url={https://arxiv.org/abs/2405.08405}
}

@inproceedings{goujaud2023fundamental,
  title={On fundamental proof structures in first-order optimization},
  author={Goujaud, Baptiste and Dieuleveut, Aymeric and Taylor, Adrien},
  booktitle={2023 62nd IEEE Conference on Decision and Control (CDC)},
  pages={3023--3030},
  year={2023},
  organization={IEEE}
}

@article{goodfellow2020generative,
  title={Generative adversarial networks},
  author={Goodfellow, Ian and Pouget-Abadie, Jean and Mirza, Mehdi and Xu, Bing and Warde-Farley, David and Ozair, Sherjil and Courville, Aaron and Bengio, Yoshua},
  journal={Communications of the ACM},
  volume={63},
  number={11},
  pages={139--144},
  year={2020},
  publisher={ACM New York, NY, USA}
}

@inproceedings{li2019robust,
  title={Robust multi-agent reinforcement learning via minimax deep deterministic policy gradient},
  author={Li, Shihui and Wu, Yi and Cui, Xinyue and Dong, Honghua and Fang, Fei and Russell, Stuart},
  booktitle={Proceedings of the AAAI conference on artificial intelligence},
  volume={33},
  number={01},
  pages={4213--4220},
  year={2019}
}

@article{Krivchenko_Gasnikov_Kovalev_2024, title={Convex-Concave Interpolation and Application of PEP to the Bilinear-Coupled Saddle Point Problem}, volume={20}, url={http://dx.doi.org/10.20537/nd241215}, DOI={10.20537/nd241215}, abstractNote={<jats:p>In this paper we present interpolation conditions for several important convex-concave function classes: nonsmooth convex-concave functions, conditions for difference of strongly-convex functions in a form that contains oracle information exclusively and smooth convex-concave functions with a bilinear coupling term. Then we demonstrate how the performance estimation problem approach can be adapted to analyze the exact worst-case convergence behavior of first-order methods applied to composite bilinear-coupled min-max problems. Using the performance estimation problem approach, we estimate iteration complexities for several first-order fixed-step methods, Sim-GDA and Alt-GDA, which are applied to smooth convex-concave functions with a bilinear coupling term.</jats:p>}, number={5}, journal={Nelineinaya Dinamika}, publisher={Izhevsk Institute of Computer Science}, author={Krivchenko, V. O. and Gasnikov, A. V. and Kovalev, D. A.}, year={2024}, pages={875–893} }

@PHDTHESIS{Taylor2017ConvexIA,
  title     = "Convex Interpolation and Performance Estimation of
First-order Methods for Convex Optimization",
  author    = "Taylor, Adrien",
  school = "Universit´e catholique de Louvain",
  year      =  2017
}

@manual{mosek,
   author = "MOSEK ApS",
   title = "The MOSEK optimization software",
   year = 2025,
   url = "http://www.mosek.com/documentation"
 }

@book{VonNeumann+Morgenstern:1944,
  author    = {von Neumann, John and Morgenstern, Oskar},
  year      = {1944},
  title     = {Theory of Games and Economic Behavior},
  publisher = {Princeton University Press},
  address   = {Princeton, NJ, USA},
  edition   = {first}
}

@article{Harker1990,
  author = {Harker, Patrick T. and Pang, Jong-Shi},
  title = {Finite-dimensional variational inequality and nonlinear complementarity problems: a survey of theory, algorithms and applications},
  journal = {Mathematical Programming},
  volume = {48},
  number = {1-3},
  pages = {161--220},
  year = {1990}
}

@ARTICLE{De_Klerk2017-fm,
  title     = "On the worst-case complexity of the gradient method with exact
               line search for smooth strongly convex functions",
  author    = "de Klerk, Etienne and Glineur, Fran{\c c}ois and Taylor, Adrien
               B",
  journal   = "Optim. Lett.",
  publisher = "Springer Science and Business Media LLC",
  volume    =  11,
  number    =  7,
  pages     = "1185--1199",
  month     =  oct,
  year      =  2017,
  language  = "en"
}

@article{doi:10.1137/16M108104X,
author = {Taylor, Adrien B. and Hendrickx, Julien M. and Glineur, Fran\c{c}ois},
title = {Exact Worst-Case Performance of First-Order Methods for Composite Convex Optimization},
journal = {SIAM Journal on Optimization},
volume = {27},
number = {3},
pages = {1283-1313},
year = {2017},
doi = {10.1137/16M108104X},
URL = { 
        https://doi.org/10.1137/16M108104X
},
eprint = { 
     https://doi.org/10.1137/16M108104X}
}

@ARTICLE{Taylor2018-gk,
  title     = "Exact worst-case convergence rates of the proximal gradient
               method for composite convex minimization",
  author    = "Taylor, Adrien B and Hendrickx, Julien M and Glineur, Fran{\c
               c}ois",
  journal   = "J. Optim. Theory Appl.",
  publisher = "Springer Science and Business Media LLC",
  volume    =  178,
  number    =  2,
  pages     = "455--476",
  month     =  aug,
  year      =  2018,
  language  = "en"
}

@article{Gupta_2023,
author = {Das Gupta, Shuvomoy and Van Parys, Bart P. G. and Ryu, Ernest K.},
title = {Branch-and-bound performance estimation programming: a unified methodology for constructing optimal optimization methods},
year = {2023},
issue_date = {Mar 2024},
publisher = {Springer-Verlag},
address = {Berlin, Heidelberg},
volume = {204},
number = {1–2},
issn = {0025-5610},
url = {https://doi.org/10.1007/s10107-023-01973-1},
doi = {10.1007/s10107-023-01973-1},
journal = {Math. Program.},
month = jun,
pages = {567–639},
numpages = {73}
}

@book{kinderlehrer2000introduction,
  title={An introduction to variational inequalities and their applications},
  author={Kinderlehrer, David and Stampacchia, Guido},
  year={2000},
  publisher={SIAM}
}

\end{document}